\documentclass[reqno]{amsart}
\usepackage{latexsym,amsmath,amssymb,amscd}
\usepackage[all]{xy}
\usepackage{enumerate}
\usepackage{hyperref} 
\makeatletter
\@addtoreset{figure}{section}
\def\thefigure{\thesection.\@arabic\c@figure}
\def\fps@figure{h,t}
\@addtoreset{table}{bsection}

\def\thetable{\thesection.\@arabic\c@table}
\def\fps@table{h, t}
\@addtoreset{equation}{section}

\makeatother

\makeatletter
\newcommand\@dotsep{4.5}
\def\@tocline#1#2#3#4#5#6#7{\relax
	\ifnum #1>\c@tocdepth 
	\else
	\par \addpenalty\@secpenalty\addvspace{#2}%
	\begingroup \hyphenpenalty\@M
	\@ifempty{#4}{%
		\@tempdima\csname r@tocindent\number#1\endcsname\relax
	}{%
		\@tempdima#4\relax
	}%
	\parindent\z@ \leftskip#3\relax \advance\leftskip\@tempdima\relax
	\rightskip\@pnumwidth plus1em \parfillskip-\@pnumwidth
	#5\leavevmode\hskip-\@tempdima #6\relax
	\leaders\hbox{$\m@th
		\mkern \@dotsep mu\hbox{.}\mkern \@dotsep mu$}\hfill
	\hbox to\@pnumwidth{\@tocpagenum{#7}}\par
	\nobreak
	\endgroup
	\fi}
\makeatother 

\newcommand{\bfi}{\bfseries\itshape}

\newtheorem{theorem}{Theorem}
\newtheorem{corollary}[theorem]{Corollary}
\newtheorem{definition}[theorem]{Definition}
\newtheorem{example}[theorem]{Example}
\newtheorem{lemma}[theorem]{Lemma}
\newtheorem{notation}[theorem]{Notation}

\newtheorem{proposition}[theorem]{Proposition}
\newtheorem{remark}[theorem]{Remark}

\numberwithin{theorem}{section}
\numberwithin{equation}{section}

\renewcommand{\1}{{\bf 1}}

\newcommand{\A}{\mathfrak{A}}
\newcommand{\I}{\mathfrak{I}}
\newcommand{\J}{\mathfrak{J}}
\renewcommand{\P}{\mathfrak{P}}
\newcommand{\Q}{\mathfrak{Q}}

\newcommand{\Ccb}{{\mathcal C}^b}

\newcommand{\Cl}{{\rm Cl}}

\newcommand{\de}{{\rm d}}
\newcommand{\ee}{{\rm e}}

\newcommand{\ev}{{\rm ev}}

\newcommand{\Hom}{{\rm Hom}}

\newcommand{\ie}{{\rm i}}

\newcommand{\Ind}{{\rm Ind}}

\newcommand{\Ker}{{\rm Ker}\,}

\newcommand{\Prim}{{\rm Prim}}

\newcommand{\CC}{{\mathbb C}}

\newcommand{\NN}{{\mathbb N}}
\newcommand{\QQ}{{\mathbb Q}}
\newcommand{\RR}{{\mathbb R}}
\newcommand{\TT}{{\mathbb T}}
\newcommand{\ZZ}{{\mathbb Z}}

\newcommand{\Bc}{{\mathcal B}}
\newcommand{\Cc}{{\mathcal C}}

\newcommand{\Hc}{{\mathcal H}}
\newcommand{\Jc}{{\mathcal J}}
\newcommand{\Kc}{{\mathcal K}}
\newcommand{\Lc}{{\mathcal L}}
\renewcommand{\Mc}{{\mathcal M}}

\newcommand{\Vc}{{\mathcal V}}

\newcommand{\Cg}{{\mathfrak C}}
\newcommand{\Jg}{{\mathfrak J}}

\newcommand{\Glimm}{{\rm Glimm}}
\newcommand{\hull}{{\rm hull}}
\newcommand{\Id}{{\rm Id}}
\newcommand{\MinPrimal}{{\rm MinPrimal}}
\newcommand{\Orc}{{\rm Orc}}
\newcommand{\Primal}{{\rm Primal}}

\newcommand{\Sub}{{\rm Sub}}

\title[Ideal spaces of Mautner groups]{A note on 
	ideal spaces of Mautner groups}

\author{Ingrid Belti\c t\u a}
\author{Daniel Belti\c t\u a}
\address{Institute of Mathematics ``Simion Stoilow'' of the Romanian Academy,
P.O. Box 1-764, Bucharest, Romania}

\email{Ingrid.Beltita@imar.ro, ingrid.beltita@gmail.com}
\email{Daniel.Beltita@imar.ro, beltita@gmail.com}

\begin{document}

\begin{abstract}
The Mautner groups are the 5-dimensional solvable Lie groups that have non-type-I factor representations. 
We show that their corresponding group $C^*$-algebras are quasi-standard 
and we describe the topology of their spaces of minimal primal ideals and Glimm ideals. 
\end{abstract}

\maketitle


\section{Introduction}

Several deep properties of the $C^*$-algebras of solvable Lie groups 
are encoded in the topology of their primitive ideal spaces, 
which, in turn, can be sometimes read off the Lie algebra of the group 
under consideration. 
This idea is particularly well illustrated by 
the class of nilpotent Lie groups, whose primitive ideal space is homeomorphic to the space of coadjoint orbits via the Kirillov correspondence. 
Beyond that class the so-called method of coadjoint orbits 
is much more difficult to use, and therefore 
there arises the challenging problem of replacing it by 
alternative approaches, involving transformation groups 
which are less general than the coadjoint action and yet, 
closer related to certain classes of Lie groups under consideration. 

This paper belongs to the line of research sketched above 
 (cf.  also \cite{BB18a}, \cite{BB18b}, \cite{BB21a}, \cite{BB21b}, and \cite{BB21c}).  
We regard the topology of the primitive ideal spaces of non-type-I solvable Lie groups from the perspective of the quasi-standard $C^*$-algebras that were introduced in \cite{AS90} and are related to quite subtle topological aspects of several ideal spaces of $C^*$-algebras. 
Minimal primal and Glimm ideal spaces of the $C^*$-algebras of generalized $(ax+b)$-groups defined by hyperbolic matrices were earlier computed in \cite[\S 5]{KST95}. 
In these notes we study these ideal spaces for another type of  generalized $(ax+b)$-groups, namely, in the case of elliptic matrices (having purely imaginary eigenvalues). 

For simplicity we focus on the classical Mautner groups, which are the 5-dim\-ensional solvable Lie groups defined as follows. 
For any $\theta\in\RR$ we denote by $G_\theta=\CC^2\rtimes_\theta\RR$ its corresponding Mautner group 
whose group operation is given 
in terms of the matrix 
\begin{equation}
\label{Dtheta}
D_\theta:=\begin{pmatrix}
2\pi\ie & \hfill 0 \\
\hfill 0 & 2\pi\ie\theta
\end{pmatrix}\in M_2(\CC)
\end{equation}
by the formula 
$$(z_1,t_1)\cdot(z_2,t_2)=(z_1+\ee^{t_1 D_\theta}z_2,t_1+t_2)$$
for all $z_1,z_2\in\CC^2$ and $t_1,t_2\in\RR$. 

Unless otherwise mentioned we assume $\theta\in\RR\setminus\QQ$ 
and we define 
\begin{equation}\label{alphatheta}
\alpha_\theta\colon \RR\times\CC^2\to\CC^2,\quad 
\alpha_\theta(t,z):=\alpha_\theta^t(z):=\ee^{tD_{\theta}}z
\end{equation}
where $D_\theta\in M_2(\CC)$ is given by \eqref{Dtheta}. 
When no ambiguity is possible, we omit $\theta$ from the notation and we write simply $\alpha$ and $\alpha^t$ instead of $\alpha_\theta$ and $\alpha_\theta^t$, respectively. 

It is known that the group $C^*$-algebra $C^*(G_\theta)$ is antiliminary, cf. \cite[Ex. 5.7]{BB21b}. 
In the present paper we obtain a precise description of the primitive ideal space of the group $C^*$-algebra  $C^*(G_\theta)$ (Theorem~\ref{closed}) and then, 
among other things, we show that this is a quasi-standard $C^*$-algebra and its spaces of minimal primal ideals and Glimm ideals are homeomorphic to the closed quadrant of the plane $[0,\infty)^2$ 
(Corollary~\ref{M}). 
We hope that this last result may shed some light on the problem of determining the possible Glimm spaces for particular classes of $C^*$-algebras, 
which was mentioned in \cite[Rem. 6.7]{LS22}.

\section{Miscellaneous preliminaries}

\subsection*{General topology}

\begin{notation}\label{ws}
	\normalfont
	For an arbitrary topological space $Y$ we denote by $\Ccb(Y)$ the set of all bounded continuous functions $f\colon Y\to\RR$. 
	We also denote by $\Cl(Y)$ the set of all closed subsets of $Y$ with its topologies $\tau_w\subseteq\tau_s$, and  we recall the following notation, cf. \cite[p. 147]{LS10}: 
	\begin{itemize}
		\item $\Lc(Y)$ the set of $L\in\Cl(Y)$ for which there exists a net in $Y$ whose set of $\tau_w$-limit points is exactly~$L$; 
		\item $\Lc'(Y):=\Lc(Y)\setminus\{\emptyset\}$; 
		\item $\Mc\Lc(Y)$ the set of all elements of $\Lc(Y)$ which are maximal with respect to the ordering given by set inclusion, in particular $\Mc\Lc(Y)\subseteq\Lc'(Y)$; 
		\item $\Mc\Lc^s(Y):=\overline{\Mc\Lc(Y)}^{\tau_s}\cap\Lc'(Y)$.
	\end{itemize}
	Unless otherwise mentioned, $\Cl(Y)$ is endowed with its Fell  topology~$\tau_s$,  which makes $\Cl(Y)$ a compact Hausdorff space 
	(cf., e.g., \cite[App. H]{Wi07}). 
\end{notation}

\begin{definition}\label{deflim}
	\normalfont 
	Let $\{A_n\}_{n\in\NN}$ be a sequence of subsets of a topological space $X$. 
	
	We define $\liminf\limits_{n\in\NN}A_n$\index{liminfA@$\liminf\limits_{n\in\NN}A_n$} as the set of all 
	points $x\in X$ with the following property: 
	For every $n\in\NN$ there exists $x_n\in A_n$ such that for 
	every $V\in\Vc_X(x)$ there exists $n_V\in\NN$ with $\{x_n\mid n\ge n_V\}\subseteq V$.
	
	We define $\limsup\limits_{n\in\NN}A_n$\index{limsupA@$\limsup\limits_{n\in\NN}A_n$} as the set of all 
	points $x\in X$ for which there exist infinitely many positive integers $1\le n_1<n_2<\cdots$ for which $x\in\liminf\limits_{k\in\NN}A_{n_k}$. 
\end{definition}

\begin{lemma}\label{L2}
	If $X$ is a topological space and $\{A_n\}_{n\in\NN}$ is a sequence of subsets of~$X$, then both $\liminf\limits_{n\to\infty}A_n$ and  $\limsup\limits_{n\to\infty}A_n$ are closed subsets of $X$. 
\end{lemma}

\begin{proof}
	See \cite[Cor. 1, page 121]{Br97}. 
\end{proof}

We now prove a localized version  of the characterization of open mappings in terms of convergent nets, cf. \cite[Prop. 1.15]{Wi07}.

\begin{lemma}
	\label{T1}
	Let $X$ and $Y$ be 1st countable topological spaces. 
	If $y\in Y$, then for any mapping $f\colon X\to Y$ the following conditions are equivalent: 
	\begin{enumerate}[{\rm(i)}]
		\item\label{T1_item1} 
		For every $x\in f^{-1}(y)$ and every $U\in\Vc_X(x)$ one has $f(U)\in\Vc_Y(y)$. 
		\item\label{T1_item2}  
		For every sequence $\{y_n\}_{n\in\NN}$ in $Y$ with $y\in\liminf\limits_{n\to\infty}\{y_n\}$, one has 
		$f^{-1}(y)\subseteq\limsup\limits_{n\to\infty}f^{-1}(y_n)$.
	\end{enumerate}
\end{lemma}

\begin{proof}
	``\eqref{T1_item1}$\Rightarrow$\eqref{T1_item2}''  
	Fix an arbitrary sequence $\{y_n\}_{n\in\NN}$ in $Y$ with $y\in\liminf\limits_{n\to\infty}\{y_n\}$ and assume 
	$f^{-1}(y)\not\subseteq\limsup\limits_{n\to\infty}f^{-1}(y_n)$, 
	hence there exists $x\in f^{-1}(y)\setminus\limsup\limits_{n\to\infty}f^{-1}(y_n)$. 
	
	One has $X\in\Vc_X(x)$, 
	hence $f(X)\in\Vc_Y(y)$ by \eqref{T1_item1}. 
	Then, by $y\in\liminf\limits_{n\to\infty}\{y_n\}$, one has 
	$y_n\in f(X)$ for all but finitely many $n\in\NN$. 
	
	By Lemma~\ref{L2}, the set $U:=X\setminus\limsup\limits_{n\to\infty}f^{-1}(y_n)$ 
	satisfies $U\in\Vc_X(x)$. 
	Since the topological space $X$ is 1st countable, it then follows that there exists a sequence of open subsets $U_1\supseteq U_2\supseteq\cdots$ of $X$ 
	with $U_1\subseteq U$ and $\bigcap\limits_{n\ge 1}
	U_n=\{x\}$. 
	
	By \eqref{T1_item1} one has $f(U_1)\in\Vc_Y(y)$ hence, 
	since $y\in\liminf\limits_{n\to\infty}\{y_n\}$, 
	there exists $n_1\ge 1$ with $y_{n_1}\in f(U_1)$. 
	Assume we have $1\le n_1<\cdots<n_k$ with $y_{n_j}\in f(U_j)$ 
	for $j=1,\dots,k$ and $y_n\in f(U_k)$ for all $n\ge n_k$. 
	Then as in the construction of $n_1$ above, 
	there exists $n_{k+1}>n_k$ with $y_n\in f(U_{k+1})$ 
	for every $n\ge n_{k+1}$. 
	We thus obtain $1\le n_1<n_2<\cdots$ with $y_{n_j}\in f(U_j)$ 
	for every $j\ge 1$. 
	In particular, for every $j\ge 1$ there exists $x_j\in U_j$ with $f(x_j)=y_{n_j}$. 
	By the properties of the sequence $\{U_j\}_{j\ge 1}$, 
	it then follows that $x\in\liminf\limits_{j\to\infty}x_j$. 
	Since $x_j\in f^{-1}(y_{n_j})$ for every $j\ge 1$, 
	we then obtain $x\in \limsup\limits_{n\to\infty}f^{-1}(y_n)$, 
	which is a contradiction with the fact that 
	$x\in f^{-1}(y)\setminus\limsup\limits_{n\to\infty}f^{-1}(y_n)$. 
	
	``\eqref{T1_item2}$\Rightarrow$\eqref{T1_item1}''  
	Assume there exist $x\in f^{-1}(y)$ and $U\in\Vc_X(x)$ with $f(U)\not\in\Vc_Y(y)$. 
	
	Since the topological space $Y$ is 1st countable, there exists 
	a sequence of open subsets $V_1\supseteq V_2\supseteq\cdots$ of $Y$ 
	with $\bigcap\limits_{n\ge 1}V_n=\{y\}$. 
	Since $f(U)\not\in\Vc_Y(y)$ and $\{V_n\}_{n\ge 1}$ is a neighborhood base of $y\in Y$, it follows that for every $n\ge 1$ one has $V_n\not\subseteq f(U)$, hence there exists $y_n\in V_n\setminus f(U)$. 
	For every $n\ge 1$ one has $y_n\in V_n$, hence $y\in\liminf_{n\to\infty}\{y_n\}$. 
	Then, by \eqref{T1_item2}, one has 
	$f^{-1}(y)\subseteq\limsup\limits_{n\to\infty}f^{-1}(y_n)$. 
	Since $x\in f^{-1}(y)$, it then follows that $x\in \limsup\limits_{n\to\infty}f^{-1}(y_n)$, 
	hence there exist $1\le n_1<n_2<\cdots$ and $x_k\in f^{-1}(y_{n_k})$ for every $k\ge 1$ with $x\in\liminf\limits_{k\to\infty}\{x_k\}$. 
	Sice $U\in\Vc_X(x)$, it then follows that there exists 
	$k_0\ge 1$ with $x_{k_0}\in U$. 
	This implies $y_{n_{k_0}}=f(x_{k_0})\in f(U)$, 
	which is a contradiction with the fact that 
	$y_n\in V_n\setminus f(U)$ for every $n\ge 1$, 
	and we are done. 
\end{proof}

\subsection*{Dynamical systems}

\begin{lemma}
	\label{Gr80Th.4.1}
	Let $G\times X\to X$, be a continuous action of a separable locally compact group $G$ on a locally compact space $X$. 
	Assume that $H\subseteq G$ is a closed subgroup for which there exists a measurable cross-section of the quotient map $G\to G/H$. 
	If $q\colon X\to G/H$ is a continuous $G$-equivariant surjective mapping and one considers the closed $H$-invariant subset $Y:=q^{-1}(\1 H)\subseteq X$, then there exists a $*$-isomorphism 
	$\Cc_0(X)\rtimes G\simeq (\Cc_0(Y)\rtimes H)\otimes \Kc(L^2(G/H))$.  
\end{lemma}

\begin{proof}
	See \cite[Th. 4.1]{Gr80}.
\end{proof}

\begin{example}\label{Gr80Th.4.1_ex1}
	\normalfont
	Let $\theta\in\RR$ and $G=(\RR,+)$ with its subgroup $H=(\ZZ,+)$. 
	For any fixed $r_1,r_2\in(0,\infty)$ and 
	$X:=r_1\TT\times r_2\TT=\{(z_1,z_2)\in\CC^2\mid \vert z_1\vert=r_1,\ \vert z_2\vert=r_2\}$ 
	we define an action of $G$ on $X$ by 
	$$\alpha_\theta \colon G\times X\to X,\quad 
	(t,(z_1,z_2))\mapsto (\ee^{2\pi\ie t}z_1,\ee^{2\pi\ie \theta t}z_2).$$
	Then the mapping 
	$$q\colon X\to G/H\simeq\{\ee^{2\pi\ie s}\mid s\in\RR\}=\TT,\quad (z_1,z_2)\mapsto z_1/r_1$$
	is $G$-equivariant. 
	Moreover, the subset $Y:=q^{-1}(\1H)=q^{-1}(1)=\{r_1\}\times r_2\TT\subseteq X$ is $H$-invariant, 
	and the corresponding action of $H=\ZZ$ on $Y$ is given by 
	$$H\times Y\to Y,\quad 
	(n,(r_1,z_2))\mapsto (r_1,\ee^{2\pi\ie \theta n}z_2) 
	=(r_1,(\ee^{2\pi\ie \theta})^n z_2).$$
	Therefore, denoting as usual by $\A_\theta$ the $C^*$-algebra generated by the rotation of angle $\theta$, 
	one clearly has $*$-isomorphisms 
	$\Cc_0(Y)\rtimes H\simeq \Cc(\TT)\rtimes\ZZ\simeq \A_\theta$. 
	(See \cite{Ri81}.) 
	Consequently, by Theorem~\ref{Gr80Th.4.1}, 
	\begin{equation}\label{Gr80Th.4.1_ex1_eq1}
	\Cc(r_1\TT\times r_2\TT)\rtimes_{\alpha_\theta}\RR\simeq \A_\theta\otimes\Kc(L^2(\TT)).
	\end{equation}
\end{example}

\subsection*{Ideal spaces of $C^*$-algebras}
For any $C^*$-algebra $\A$ we denote by $\Id(\A)$ its set of closed two-sided ideals and $\Prim(\A)\subseteq\Id(\A)$ is the set of all primitive ideals, i.e., the kernels of irreducible $*$-representations of $\A$.

There are two topologies $\tau_w\subseteq\tau_s$ on $\Id(\A)$ which can be described as follows.
Use the quotient maps $\A\to \A/\I$, $a\mapsto a+\I$, for all $\I\in\Id(\A)$ 
to define the family of functions 
$$\varphi_a\colon\Id(\A)\to\RR^+,\quad \varphi_a(\I):=\Vert a+\I\Vert \text{ for all }\I\in\Id(\A)\text{ and }a\in \A.$$
Then $\tau_s$ (respectively, $\tau_w$) is the weakest topology on $\Id(\A)$ with respect to which all the functions $\varphi_a$ for $a\in \A$ are continuous (respectively, lower semi-continuous), cf. \cite[p. 84]{AS93}. 

{\bfi Unless otherwise mentioned, $\Prim(\A)$ is endowed with the topology $\tau_w$}, 
(which coincides with the Jacobson topology of $\Prim(\A)$, 
whose closed sets are $\hull(\I)$ for $\I\in\Id(\A)$, cf. Lemma~\ref{primals} below). 

\begin{remark}
	\normalfont
The topological space $(\Id(\A),\tau_s)$ is compact Hausdorff, since it is homeomorphic to $\Cl(\Prim(\A))$ with its Fell topology via the hull/kernel mappings, 
cf. Lemma~\ref{primals} below. 
Compare also Notation~\ref{ws} for general topological spaces. 
\end{remark}

\begin{definition}[Primal ideals]
\normalfont 
We introduce the following subsets of $\Id(\A)$: 
\begin{align}
\Primal(\A) & :=\overline{\Prim(\A)}^{\tau_w} \\
\Primal'(\A) & :=\Primal(\A)\setminus\{\A\} 
\end{align}
\end{definition}

\begin{definition}[Minimal primal ideals]
\normalfont 
	We introduce the following subsets of $\Primal(\A)$ and $\Primal'(\A)$, respectively: 
\begin{align}
\MinPrimal(\A) & :=\text{the minimal elements of }\Primal(\A)\\
\tau &:=\tau_w\vert_{\MinPrimal(\A)}=\tau_s\vert_{\MinPrimal(\A)} \quad  
(\text{cf. \cite[Cor. 4.3(a)]{A87}}) \\
\Sub(\A) & :=\overline{\MinPrimal(\A)}^{\tau_s}\cap \Primal'(\A) 
=\overline{\MinPrimal(\A)}^{\tau_s}\setminus\{\A\}
\end{align}
\end{definition}

\begin{lemma}\label{primals}
	For every $C^*$-algebra $\A$, the homeomorphism 
	$$\hull\colon\Id(\A)\to\Cl(\Prim(\A)), \ \hull(\Jc):=\{\P\in\Prim(\J)\mid \J\subseteq\P\}\simeq\Prim(\A/\J)$$
	and its inverse 
	$$\ker\colon \Cl(\Prim(\A))\to\Id(\A),\quad \ker(L):=\bigcap_{\P\in L}\P$$
	give rise to the commutative diagram
	$$\xymatrix{
		\MinPrimal(\A) \ar@{^{(}->}[r] \ar[d]_{\hull} 
		& \Sub(\A) \ar@{^{(}->}[r] \ar[d]_{\hull}
		& \Primal'(\A) \ar@{^{(}->}[r] \ar[d]_{\hull} 
		& \Primal(\A) \ar[d]_{\hull}\\
		\Mc\Lc(\Prim(\A))  \ar@{^{(}->}[r]  \ar@<-1ex>[u]_{\ker}
		& \Mc\Lc^s(\Prim(\A)) \ar@{^{(}->}[r] \ar@<-1ex>[u]_{\ker}
		& \Lc'(\Prim(\A)) \ar@{^{(}->}[r] \ar@<-1ex>[u]_{\ker}
		& \Lc(\Prim(\A)) \ar@<-1ex>[u]_{\ker}
	} $$
	whose vertical arrows are homeomorphisms.  
\end{lemma}

\begin{proof}
	See \cite[Cor. 1.3 and Lemma 1.4]{AS93}, \cite[Prop. 3.2]{AB86}, 
	\cite[p. 46]{KST95}, and also \cite{LS10}. 
\end{proof}

\begin{remark}
\normalfont
In the setting of Lemma~\ref{primals} we note that $\Mc\Lc^s(\A)$ is a closed subset of the locally compact Hausdorff space $\Lc'(\A)$, cf. \cite[end of \S 1]{LS10}, and consequently $\Sub(\A)$ is a closed subset of the locally compact Hausdorff space $\Primal'(\A)$. 
\end{remark}

\begin{definition}[Glimm ideals]
\normalfont 
The set of \emph{Glimm ideals} of the $C^*$-algebra $\A$ 
is the image of the mapping 
\begin{align*}
q\colon & \Prim(\A)\to\Glimm(\A) (\subseteq\Id(\A))\\ 
& \hull(q(\P)):=\{\P'\in\Prim(\A)\mid(\forall f\in\Ccb(\Prim(\A)))\ f(\P')=f(\P)\}
\end{align*}
and we define 
$\tau_q$ as the quotient topology of $\Glimm(\A)$ arising from $(\Prim(\A),\tau_w)$ via $q$. 
\end{definition}

\begin{definition}
\normalfont 
We define a binary relation $\sim$ on $\Prim(\A)$ by 
$$\P\sim \Q\iff (\forall C\in \Cl(\Prim(\A)))\quad \text{ either } \{\P,\Q\}\subseteq C \text{ or }\{\P,\Q\}\cap C=\emptyset.$$
We say $\A$ is \emph{quasi-standard} if  
$\sim$ is an open equivalence relation. 
\end{definition}

\begin{remark}
	\normalfont 
	This binary relation is symmetric and reflexive, but in general not transitive. 
	The failure of $\sim$ from being transitive is measured by the so-called \emph{connecting order} $\Orc(\A)\in\{1,2,3,\dots\}\cup\{\infty\}$ 
	and we have $\Orc(\A)=1$ if and only if $\sim$ is transitive, hence is an equivalence relation, cf. \cite[\S 2]{S93}.
\end{remark}

\begin{remark}\label{quasi}
	\normalfont
	We collect a few remarks on a $C^*$-algebra $\A$, which are needed in the proof of Corollary~\ref{M}.
	\begin{enumerate}[{\rm(i)}]
		\item\label{quasi_item1} $\A$ is quasi-standard iff $(\MinPrimal(\A),\tau)=(\Glimm(\A),\tau_q)$ by \cite[Th. 3.3]{AS90} and also \cite[p. 48]{KST95}.
		\item\label{quasi_item2} If $\A$ is quasi-standard, then $\MinPrimal(\A)=\Sub(\A)=\Glimm(\A)$ as sets and as topological spaces, cf. \cite[p. 238]{AKS15}.
		\item\label{quasi_item3} If $\A$ is not quasi-standard and  $\MinPrimal(\A)=\Glimm(\A)$ as sets, then $\MinPrimal(\A)\subsetneqq \Sub(\A)$, cf. again \cite[p. 238]{AKS15}.
	\end{enumerate}
\end{remark}

\section{The space of primitive ideals of $C^*(G_\theta)$}

\subsection{Williams parameterization of $\Prim(X\rtimes_\alpha A)$}

\begin{notation}[induced representations]
	\normalfont 
	Let $A$ be a locally compact abelian group. 
	For every closed subgroup $B\subseteq A$ and every $\chi\in\widehat{B}:=\Hom(B,\TT)$ we define  
	$$\Ccb(A,\chi):=\{\varphi\in\Ccb(A)\mid 
	(\forall a\in A,b\in B)\ \varphi(ab)=\chi(b)^{-1}\varphi(a)\}$$
	and 
	$$(\forall \varphi\in\Ccb(A,\chi))\quad 
	\Vert \varphi\Vert_{\chi} :=\Bigl(\int\limits_{A/B}\vert\varphi(aB)\vert^2\de(aB)\Bigr)^{1/2}\in[0,\infty].$$
	We also denote by $L^2(A,\chi)$ the Hilbert space obtained as the completion of the pre-Hilbert space $\{\varphi\in\Ccb(A,\chi)\mid \Vert\varphi\Vert_\chi<\infty\}$ with respect to the norm $\Vert\cdot\Vert_\chi$. 
	Then the regular representation $\lambda_A\colon A\to\Bc(\Ccb(A))$, $\lambda_A(a)\varphi:=\varphi(a+\cdot)$, leaves invariant the space $\Ccb(A,\chi)$ hence it gives rise to a unitary representation 
	in the Hilbert space $L^2(A,\chi)$, to be denoted as 
	$$\Ind_B^A(\chi)\colon A\to\Bc(L^2(A,\chi))$$
	(the unitary \emph{representation of $A$ induced from} $\chi\in\widehat{B}$).
\end{notation}

\begin{notation}
	\normalfont
	We recall that a \emph{$C^*$-dynamical system} $(\Cg,A,\alpha)$ consists of a $C^*$-algebra $\Cg$, a locally compact group $A$, and a continuous right action of $A$ by $*$-automorphisms of~$\Cg$, denoted $\alpha\colon \Cg\times A\to\Cg$, $(\xi,a)\mapsto \xi.a$. 
	The \emph{crossed product} corresponding to the $C^*$-dynamical system $(\Cg,A,\alpha)$ is denoted by $\Cg\rtimes_\alpha A$ unless otherwise mentioned. 
	
	Let $\Id^A(\Cg)$ be the set of all 
	closed 2-sided ideals $\Jg\subseteq\Cg$ satisfying $\Jg.A\subseteq\Jg$.  If $\Jg\in\Id^A(\Cg)$, then the action of $A$ on $\Jg$ is also denoted by $\alpha$, hence the corresponding 
	crossed product is consequently denoted by $\Jg\rtimes_\alpha A$. 
	The natural action of $A$ on $\Cg/\Jg$ is denoted by $\alpha^\Jg$, 
	and one has the short exact sequence 
	\begin{equation}\label{exact}
	0\to\Jg\rtimes_\alpha A\hookrightarrow\Cg\rtimes_\alpha A\to 
	(\Cg/\Jg)\rtimes_{\alpha^\Jg}A\to0. 
	\end{equation}
	(See \cite[Prop. 3.19]{Wi07}.) 
	
	A \emph{covariant representation} of the $C^*$-dynamical system $(\Cg,A,\alpha)$ 
	is a pair $(\mu,\pi)$ consisting of a $*$-representation $\mu\colon\Cg\to\Bc(\Hc)$ and a strongly continuous unitary representation $\pi\colon A\to\Bc(\Hc)$ satisfying $\mu(\xi.a)=\pi(a)\mu(\xi)\pi(a)^{-1}$ for all $\xi\in\Cg$ and $a\in A$. 
	Such a covariant representation gives rise to a unique $*$-representation 
	$\mu\rtimes\pi\colon \Cg\rtimes_\alpha A\to\Bc(\Hc)$ called its \emph{integrated form}, satisfying 
	\begin{equation}\label{integrated}
	(\mu\rtimes\pi)(f)=\int\limits_A\mu(f(a))\pi(a)\de a
	\end{equation}
	for all $f\in L^1(A,\Cg)$. 
	The covariant representation $(\mu,\pi)$ is called \emph{irreducible} if $\mu(\Cg)'\cap\pi(A)'=\CC\1$ or, equivalently, if  
	$\mu\rtimes\pi$ is an irreducible $*$-representation. 
\end{notation}

\begin{lemma}\label{param0}
	If $(\mu,\pi)$ is a covariant representation of the $C^*$-dynamical system  $(\Cg,A,\alpha)$ and one denotes $\Jg:=\Ker\mu\subseteq\Cg$, 
	then the following assertions hold: 
	\begin{enumerate}[{\rm(i)}]
		\item\label{param0_item1} $\Jg\in\Id^A(\Cg)$ and $\Jg\rtimes_\alpha A\subseteq\Ker(\mu\rtimes\pi)$. 
		\item\label{param0_item2} If moreover 
		the $C^*$-algebra $(\Cg/\Jg)\rtimes_{\alpha^\Jg}A$ is simple, then either $\Jg\rtimes_\alpha A=\Ker(\mu\rtimes\pi)$ or $\mu\rtimes\pi=0$. 
	\end{enumerate}
\end{lemma}

\begin{proof}
	\eqref{param0_item1} 
	See \cite[Lemma 6.16]{Wi07}. 
	In some more detail, 
	since $\Jg=\Ker\mu$, it then follows by \eqref{integrated} that $L^1(A,\Jg)\subseteq\Ker(\mu\rtimes\pi)$. 
	Then $\Jg\rtimes_\alpha A\subseteq\Ker(\mu\rtimes\pi)$ 
	since $L^1(A,\Jg)$ is dense in $\Jg\rtimes_\alpha A$.
	
	\eqref{param0_item2} 
	Since the $C^*$-algebra $(\Cg/\Jg)\rtimes_{\alpha^\Jg}A$ is simple, it follows by the exact sequence~\eqref{exact} that $\Jg\rtimes_\alpha A$ is a maximal ideal of $\Cg\rtimes_\alpha A$. 
	Then, by \eqref{param0_item1}, 
	we obtain either $\Jg\rtimes_\alpha A=\Ker(\mu\rtimes\pi)$ or $\Ker(\mu\rtimes\pi)=\Cg\rtimes_\alpha A$, 
	and this completes the proof. 
\end{proof}

\begin{lemma}\label{param1}
	Let $A$ be a locally compact abelian group, $X$ be a locally compact space, 
	and assume that both $A$ and $X$ are second countable. 
	We define the representation by multiplication operators 
	$$M\colon L^\infty(A)\to\Bc(\Ccb(A)),\quad M(\varphi)f:=\varphi f$$
	and the evaluation functionals
	$$(\forall x\in X)\quad \ev_x\colon\Cc_0(X)\to\CC,\ \ev_x(\xi):=\xi(x).$$
	If 
	$$\alpha\colon A\times X\to X,\quad (a,x)\mapsto\alpha(a,x)=\alpha_x(a)=a.x$$
	is a continuous action of $A$ on $X$, then we also define for every $x\in X$ 
	its corresponding stability group 
	$$A_x:=\{a\in A\mid a.x=x\}$$
	and the operators 
	$$\alpha_x^*\colon \Cc_0(X)\to\Ccb(A),\quad \xi\mapsto\xi\circ\alpha_x$$
	and 
	$$M^\alpha_x\colon\Cc_0(X)\to\Bc(\Ccb(A)),\quad M^\alpha_x:= M\circ\alpha_x^*.$$
	Then the following assertions hold true. 
	\begin{enumerate}[{\rm(i)}]
		\item\label{param1_item1} 
		One has $M^\alpha_x(\Cc_0(X))\Ccb(A,\chi)\subseteq \Ccb(A,\chi)$ for every $\chi\in\widehat{A_x}=\Hom(A_x,\TT)$ and  $x\in X$. 
		This gives rise to a $*$-representation 
		$$\widetilde{M}^\alpha_x\colon \Cc_0(X)\to\Bc(L^2(A,\chi))$$
		for which the pair $(\widetilde{M}^\alpha_x,\Ind_{A_x}^A(\chi))$ is an irreducible covariant representation of the $C^*$-dynamical system $(\Cc_0(X),A)$. 
		\item\label{param1_item2}  
		The mapping 
		$$\Phi\colon X\times\widehat{A}\to\Prim(C^*(X\rtimes_\alpha A)),\quad (x,\tau)\mapsto\Ker(\widetilde{M}^\alpha_x\rtimes\Ind_{A_x}^A(\tau\vert_{A_x}))$$
		is surjective, continuous, and open. 
		\item\label{param1_item3}  
		For any $(x_1,\tau_1),(x_2,\tau_2)\in X\times\widehat{A}$ one has 
		$$\Phi(x_1,\tau_1)=\Phi(x_2,\tau_2) 
		\iff\begin{cases} 
		1. & \overline{A.x_1}=\overline{A.x_2} \ (\Rightarrow A_{x_1}=A_{x_2})\\
		2. & \tau_1\vert_{A_{x_1}}=\tau_2\vert_{A_{x_1}}.
		\end{cases}$$
		\item\label{param1_item4}
		If we define the equivalence relation $\sim$ on $X\times\widehat{A}$ by 
		$$(x_1,\tau_1)\sim(x_2,\tau_2)\iff \Phi(x_1,\tau_1)=\Phi(x_2,\tau_2)$$ 
		and we denote by $[(x,\tau)]:=\Phi^{-1}(\Phi(x,\tau))\in(X\times\widehat{A})/\sim$ the equivalence class of an arbitrary $(x,\tau)\in X\times\widehat{A}$, 
		then the mapping 
		$$\widetilde{\Phi}\colon (X\times\widehat{A})/\sim\ \to\Prim(C^*(X\rtimes_\alpha A)),\quad [(x,\tau)]\mapsto\Phi(x,\tau)$$
		is a well-defined homeomorphism. 
	\end{enumerate}
\end{lemma}

\begin{proof}
	\eqref{param1_item1} See \cite[Prop. 5.4, Prop. 8.24, and Prop. 8.27]{Wi07}.
	
	\eqref{param1_item2} See \cite[Th. 8.39]{Wi07}.
	
	\eqref{param1_item3} See \cite[Th. 8.39]{Wi07}. 
	The implication $\overline{A.x_1}=\overline{A.x_2} \Rightarrow A_{x_1}=A_{x_2}$ follows by \cite[Lemma 8.34]{Wi07}.
	
	\eqref{param1_item4} See \cite[Th. 8.39]{Wi07}. 
\end{proof}

\begin{definition}\label{param2}
	\normalfont
	In the setting of Lemma~\ref{param1},  
	the homeomorphism $\widetilde{\Phi}$ will be hereafter called as the \emph{Williams parameterization} of $\Prim(C^*(X\rtimes_\alpha A))$. 
\end{definition}

\begin{remark}\label{param2.6}
	\normalfont 
	In the setting of Lemma~\ref{param1},  
	if $x_0\in X$ is a point whose orbit $A.x_0$ is a closed subset of $X$, then for every $\tau\in\widehat{A}$ there exists a $*$-isomorphism 
	\begin{equation}
	\label{param2.6_eq1}
	(\Cc_0(X)\rtimes_\alpha A)/\Phi(x_0,\tau)\simeq \Kc(L^2(A/A_{x_0})).
	\end{equation}
	This can be obtained along the lines of the proof of \cite[Prop. 7.31]{Wi07}, but we provide the details since we are not aware of any precise reference for this result. 
	
	In fact, since $A.x_0$ is a closed subset of $X$, it easily follows by the definition of $\widetilde{M}^\alpha_{x_0}=:\mu$ in Lemma~\ref{param1} that 
	$$\Ker \widetilde{M}^\alpha_{x_0}=\{\xi\in \Cc_0(X)\mid \xi\vert_{A.x_0}=0\}=:\Jg.$$
	It follows by Lemma~\ref{param0}\eqref{param0_item1} that 
	$\widetilde{M}^\alpha_x\rtimes_\alpha\Ind_{A_{x_0}}^A(\tau\vert_{A_{x_0}})$ 
	factors through a representation 	$\overline{\mu}\rtimes_\alpha\Ind_{A_{x_0}}^A(\tau\vert_{A_{x_0}})$  of $(\Cc_0(X)/\Jg)\rtimes_{\alpha^\Jg} A$ 
	with $\Ker\overline{\mu}=\{0\}$. 
	Since $\widetilde{M}^\alpha_x\rtimes_\alpha\Ind_{A_{x_0}}^A(\tau\vert_{A_{x_0}})$ is irreducible by Lemma~\ref{param1}, it follows that 
	$\overline{\mu}\rtimes_\alpha\Ind_{A_{x_0}}^A(\tau\vert_{A_{x_0}})$ is irreducible as well. 
	
	On the other hand, $\Cc_0(X)/\Jg\simeq\Cc_0(A.x_0)$ 
	and moreover, the mapping $A/A_{x_0}\to A.x_{0}$, $a\mapsto a.x_0$, is a homeomorphism since $A.x_0$ is a closed subset of $X$, hence locally closed. 
	Therefore one has $*$-isomorphisms 
	\begin{align}
	(\Cc_0(X)/\Jg)\rtimes_{\alpha^\Jg} A
	& \simeq \Cc_0(A/A_{x_0})\rtimes A \nonumber \\
	&\simeq C^*(A_{x_0})\otimes \Kc(L^2(A/A_{x_0})) \nonumber  \\
	\label{param2.6_eq2}
	& \simeq \Cc_0(\widehat{A_{x_0}})\otimes \Kc(L^2(A/A_{x_0}))
	\end{align}
	where the next-to-last $*$-isomorphism is a very special case of Theorem~\ref{Gr80Th.4.1} when $G:=A$, $H:=A_{x_0}$, and $q$ is the identity map of $G/H$. 
	Now, since $\overline{\mu}\rtimes_\alpha\Ind_{A_{x_0}}^A(\tau\vert_{A_{x_0}})$  is an irreducible $*$-representation of  $(\Cc_0(X)/\Jg)\rtimes_{\alpha^\Jg} A$ whose image is $*$-isomorphic to $(\Cc_0(X)\rtimes_\alpha A)/\Phi(x_0,\tau)$, 
	one obtains \eqref{param2.6_eq1} by \eqref{param2.6_eq2}.
\end{remark}

\begin{remark}\label{param2.7}
	\normalfont 
	In the setting of Lemma~\ref{param1},  
	assume that $x_0\in X$ is a fixed point with respect to the action of $A$ or, equivalently, $A_{x_0}=A$. 
	Then for every $\tau\in\widehat{A}$ one has 
	$(\Cc_0(X)\rtimes_\alpha A)/\Phi(x_0,\tau)\simeq\CC$ by Remark~\ref{param2.6}, 
	but this conclusion can be directly obtained as follows. 
	
	If $\tau\in\widehat{A}$ then 
	$\dim L^2(A,\tau\vert_{A_{x_0}})=1$; 
	more exactly, the mapping $\varphi\mapsto\varphi(\1)$ gives a unitary operator $L^2(A,\tau\vert_{A_{x_0}})\to\CC$. 
	This easily implies that  $\Ind_{A_{x_0}}^A(\tau\vert_{A_{x_0}})=\tau\colon A\to\CC$ 
	and, moreover, the $*$-representation $\widetilde{M}^\alpha_{x_0}\colon \Cc_0(X)\to\Bc(L^2(A,\tau))$ takes the form 
	$\widetilde{M}^\alpha_{x_0}=\ev_{x_0}\colon \Cc_0(X)\to\CC$. 
	Therefore one can write 
	$$\widetilde{M}^\alpha_x\rtimes_\alpha\Ind_{A_{x_0}}^A(\tau\vert_{A_{x_0}}) 
	\colon \Cc_0(X)\rtimes_\alpha A\to\CC$$
	and $(\widetilde{M}^\alpha_x\rtimes\Ind_{A_x}^A(\tau\vert_{A_{x_0}}))(f)=\int\limits_A\ev_{x_0}(f(a))\tau(a)\de a$ for every $f\in L^1(A,\Cg)$ by \eqref{integrated}. 
	In particular, one has a $*$-isomorphism $(\Cc_0(X)\rtimes_\alpha A)/\Phi(x_0,\tau)\simeq\CC$ for every $\tau\in\widehat{A}$. 
\end{remark}

\begin{lemma}\label{param3}
	In the setting of Lemma~\ref{param1}, the following assertions hold.  
	\begin{enumerate}[{\rm(i)}] 
		\item\label{param3_item1} 
		The quotient mapping 
		$$q\colon X\times\widehat{A}\to (X\times\widehat{A})/\sim,\quad 
		q(x,\tau):=[(x,\tau)]$$
		is open. 
		
		\item\label{param3_item2}  
		If $\lim\limits_{n\in\NN}x_n=x$ in $X$ and  $\lim\limits_{n\in\NN}\tau_n=\tau$ in $\widehat{A}$, 
		then 
		$[(x,\tau)]\in\liminf\limits_{n\in\NN}\{[(x_n,\tau_n)]\}$ in $(X\times\widehat{A})/\sim$. 
		\item\label{param3_item3}  
		If $[(x,\tau)]\in\liminf\limits_{n\in\NN}\{[(x_n,\tau_n)]\}$ in $(X\times\widehat{A})/\sim$, then there exists a sequence of natural numbers $n_1<n_2<\cdots$ such that for every $k\in\NN$ there exist $x'_k\in X$ and $\tau'_k\in\widehat{A}$ with $[(x_{n_k},\tau_{n_k})]=[(x'_k,\tau'_k)]$ 
		and moreover $\lim\limits_{k\in\NN}x'_k=x$ and  $\lim\limits_{k\in\NN}\tau'_k=\tau$. 
	\end{enumerate}
\end{lemma}

\begin{proof}
	\eqref{param3_item1} 
	This follows by Lemma~\ref{param1} (\eqref{param1_item2}--\eqref{param1_item4}). 
	
	\eqref{param3_item2} 
	This holds true since the quotient mapping $q$ is continuous. 
	
	\eqref{param3_item3}  
	This assertion follows by \eqref{param3_item1} along with Lemma~\ref{T1}.
\end{proof}

\subsection{Parameterization of $\Prim\,C^*(G_\theta)$}

Here we fix $\theta\in\RR\setminus\QQ$ and we specialize Lemma~\ref{param1} for $A=(\RR,+)$, $X=\CC^2$, and 
$$\alpha\colon A\times X\to X,\quad (t,x)\mapsto \ee^{tD_\theta}x$$
where $D_\theta$ is given by \eqref{Dtheta}. 
We also perform the identification $\RR\simeq\widehat{A}$ via the duality pairing 
$$\RR\times \RR\to\TT,\quad (\tau,t)\mapsto\ee^{2\pi\ie \tau t}.$$
It follows that for arbitrary 
$x=\begin{pmatrix} z_1\\ z_2\end{pmatrix}\in \CC^2= X$ and $\tau\in\RR\simeq\widehat{A}$
their corresponding isotropy group $A_x\subseteq\RR$ and equivalence class $[(x,\tau)]\in(X\times\widehat{A})/\sim$ can be summarized like this
\begin{equation}\label{tab}
\text{
	\begin{tabular}{ | c | c | c |}
	\hline
	$x=\begin{pmatrix} z_1\\ z_2\end{pmatrix}$, $\tau\in\RR$ 	& $A_x$ & $[(x,\tau)]\subseteq\CC^2\times\RR$  \\ 
	\hline
	$z_1z_2\ne0$  	& $\{0\}$ & $\bigl((r_1\TT)\times(r_2\TT)\bigr)\times\RR$\\
	\hline
	$z_1\ne0=z_2$    	& $\ZZ$ &  $\bigl((r_1\TT)\times\{0\}\bigr)\times(\tau+\ZZ)$\\ 
	\hline
	$z_1=0\ne z_2$  	& $\frac{1}{\theta}\ZZ$ &  $\bigl(\{0\}\times(r_2\TT)\bigr)\times(\tau+\theta\ZZ)$\\ 
	\hline
	$z_1=z_2=0$ & $\RR$ &  $\{(0,0,\tau)\}$ \\ 
	\hline
	\end{tabular}
}
\end{equation}
where we have denoted $r_j:=\vert z_j\vert\in[0,\infty)$, hence $r_j\TT=\{w\in\CC\mid\vert w\vert=r_j\}$ for $j=1,2$. 
We computed $[(x,\tau)]$ in \eqref{tab}, using the fact that  $\overline{A.x}=(r_1\TT)\times(r_2\TT)$ since $\theta\in\RR\setminus\QQ$. 

In order to describe the topology of $(X\times\widehat{A})/\sim$ we need the open sets 
$$\emptyset=\widetilde{D_0}\subset \widetilde{D_1}\subset \widetilde{D_2}\subset \widetilde{D_3}=X\times\widehat{A}$$
where\footnote{We denote $\CC^\times:=\CC\setminus\{0\}$.} 
$$\widetilde{D_1}:=(\CC^\times)^2\times\RR$$
and 
$$\widetilde{D_2}:=\bigl(\CC^2\setminus\{0\}\bigr)\times\RR.$$
Let 
$$q\colon X\times\widehat{A}\to (X\times\widehat{A})/\sim,\quad 
q(x,\tau):=[(x,\tau)]$$
be the canonical quotient map, whose fibers are described in the third column of~\eqref{tab}. 
We also denote $D_j:=q(\widetilde{D_j})$ and $\Gamma_j:=D_j\setminus D_{j-1}$ for $j=1,2,3$, 
and we note that $\Gamma_2$ has two connected components $\Gamma_2'$ and $\Gamma_2''$, hence $\Gamma_2=\Gamma_2'\sqcup\Gamma_2''$, where 
\begin{align*}
\Gamma_2'&:=\{[(x,\tau)]\mid x\in\CC^\times\times\{0\}\},\\
\Gamma_2''&:=\{[(x,\tau)]\mid x\in\{0\}\times\CC^\times\}. 
\end{align*}
We also need the following lemma. 

\begin{lemma}\label{sep}
	The sets $\Gamma_1$,
	$\Gamma_1\sqcup\Gamma_1'$, and $\Gamma_1\sqcup\Gamma_1''$ 
	are open in $(X\times\widehat{A})/\sim$. 
\end{lemma}

\begin{proof}
	Let $p_1,p_2\colon\CC^2\to\CC$ be the Cartesian projections. 
	Then it is easily seen that the functions 
	\begin{equation}\label{sep_proof_eq1}
	f_1,f_2\colon (X\times\widehat{A})/\sim\ \to[0,\infty),\quad 
	f_j([(x,\tau)]):=\vert p_j(x)\vert \text{ for }j=1,2
	\end{equation}
	are well defined, continuous, and one has 
	$\Gamma_1\sqcup\Gamma_1'=f_1^{-1}(0,\infty)$ 
	and $\Gamma_1\sqcup\Gamma_1''=f_2^{-1}(0,\infty)$. 
\end{proof}

Now the topology of $(X\times\widehat{A})/\sim$ can be described as follows. 

\begin{theorem}
	\label{closed}
	A subset $F\subseteq(X\times\widehat{A})/\sim$ is closed if and only if it satisfies the following conditions: 
	\begin{enumerate}[{\rm(i)}]
		\item\label{closed_item1} The set $F\cap \Gamma_j$ is closed in the relative topology of the subset $\Gamma_j\subseteq(X\times\widehat{A})/\sim$ for $j=1,2,3$. 
		\item\label{closed_item2} 
		If $[(x_0,\tau_0)]\in \Gamma_2\cup\Gamma_3$ is an accumulation point of $F\cap\Gamma_1$, then $[(x_0,\tau)]\in F$ for all $\tau\in\RR$. 
		\item\label{closed_item3} 
		If $[(0,\tau_0)]\in \Gamma_3$ is an accumulation point of $F\cap\Gamma_2'$, then $[(0,\tau)]\in F$ for all $\tau\in\tau_0+\ZZ$. 
		\item\label{closed_item4} 
		If $[(0,\tau_0)]\in \Gamma_3$ is an accumulation point of $F\cap\Gamma_2''$, then $[(0,\tau)]\in F$ for all $\tau\in\tau_0+\theta\ZZ$. 
	\end{enumerate}
\end{theorem}

\begin{proof}
	We first show that \eqref{closed_item1}--\eqref{closed_item2} are necessary.

	\eqref{closed_item1} This is clear. 
	
	\eqref{closed_item2} 
	Since  
	$[(x_0,\tau_0)]$ is an accumulation point of $F\cap\Gamma_1$, 
	we can find $(y_n,\sigma_n)\in\widetilde{D_1}\cap q^{-1}(F)$ for all $n\in\NN$ with $q(x_0,\tau_0)\in\lim\limits_{n\in\NN}\{q(y_n,\sigma_n)\}$. 
	Then, by Lemma~\ref{param3}\eqref{param3_item3}, 
	after selecting a suitable subsequence and relabeling, we may assume $x_0=\lim\limits_{n\in \NN}y_n$. 
	For arbitrary $\tau\in\RR$, it then follows by Lemma~\ref{param3}\eqref{param3_item2} that $q(x_0,\tau)\in\lim\limits_{n\in\NN}\{q(y_n,\tau)\}$. 
	Since $y_n\in(\CC^\times)^2$, 
	we have that $q(y_n,\tau)=q(y_n,\sigma_n)$ by~\eqref{tab}, 
	hence $q(y_n,\tau)\in F$, 
	for every $n\in\NN$, and then $q(x_0,\tau)\in F$ since $F$ is closed.

	\eqref{closed_item3} 
	Since $[(x_0,\tau_0)]\in\Gamma_3=D_3\setminus D_2$, one has $x_0=0\in\CC^2$. 
	On the other hand, if  
	$[(0,\tau_0)]=[(x_0,\tau_0)]$ is an accumulation point of $F\cap\Gamma_2'$, 
	we can select 
	$y_n\in\CC^\times\times\{0\}$ and $\sigma_n\in\RR$ 
	for every $n\in\NN$ with $q(x_0,\tau_0)\in\lim\limits_{n\in\NN}\{q(y_n,\sigma_n)\}$. 
	By Lemma~\ref{param3}\eqref{param3_item3}, 
	after selecting again a suitable subsequence and relabeling, we may assume $\lim\limits_{n\in \NN}y_n=0$ 
	and $\lim\limits_{n\in \NN}\sigma_n=\tau_0$. 
	For arbitrary $k\in\ZZ$, it then follows by Lemma~\ref{param3}\eqref{param3_item2} that $q(0,\tau_0+k)\in\lim\limits_{n\in\NN}\{q(y_n,\sigma_n+k)\}$. 
	Since $y_n\in\CC^\times\times\{0\}$, 
	one has $q(y_n,\sigma_n+k)=q(y_n,\sigma_n)$ by~\eqref{tab}, 
	hence $q(y_n,\sigma_n+k)\in F$, 
	for every $n\in\NN$, and then $q(0,\tau_0+k)\in F$, since $F$ is closed. 
	
	\eqref{closed_item4} This is similar to \eqref{closed_item3}, and therefore we omit the details. 
	
	Conversely, we now assume that $F$ satisfies \eqref{closed_item1}--\eqref{closed_item4} and we prove that $F$ is closed. 
	To this end we must check that if $[(x_0,\tau_0)]\in(X\times\widehat{A})/\sim$ is an accumulation point of $F$, then $[(x_0,\tau_0)]\in F$. 
	One has the partition 
	\begin{equation}
	\label{closed_proof_eq1}
	(X\times\widehat{A})/\sim\ =\Gamma_1\sqcup(\Gamma_2'\sqcup\Gamma_2'')\sqcup\Gamma_3
	\end{equation}
	and we discuss separately the cases that may occur. 
	
	Case 1: $[(x_0,\tau_0)]\in \Gamma_1$. 
	Here we note that $\Gamma_1=D_1$, and $D_1\subseteq(X\times\widehat{A})/\sim$ is an open subset. 
	Therefore, since $[(x_0,\tau_0)]$ is an accumulation point of $F$, 
	then $[(x_0,\tau_0)]$ is actually an accumulation point of $F\cap \Gamma_1$, 
	hence $[(x_0,\tau_0)]\in F\cap \Gamma_1$ by~\eqref{closed_item1}. 
	
	Case 2: $[(x_0,\tau_0)]\in \Gamma_2$. 
	It follows by \eqref{closed_proof_eq1} and Lemma~\ref{sep} that $\Gamma_3$ is closed, hence 
	$[(x_0,\tau_0)]$ cannot be an accumulation point of $F\cap\Gamma_3$. 
	Therefore $[(x_0,\tau_0)]$ is an accumulation point of $F\cap (\Gamma_1\cup\Gamma_2)$. 
	Now, if $[(x_0,\tau_0)]$ is an accumulation point of $F\cap \Gamma_1$, then 
	$[(x_0,\tau_0)]\in F$ by~\eqref{closed_item2}, 
	while if $[(x_0,\tau_0)]$ is an accumulation point of $F\cap \Gamma_2$, then 
	$[(x_0,\tau_0)]\in F$ by~\eqref{closed_item1}. 
	
	Case 3: $[(x_0,\tau_0)]\in \Gamma_3$, that is, $x_0=0$.  
	If $[(x_0,\tau_0)]$ is an accumulation point of $F\cap \Gamma_1$, then 
	$[(x_0,\tau_0)]\in F$ by~\eqref{closed_item2}, 
	while if $[(x_0,\tau_0)]$ is an accumulation point of $F\cap \Gamma_3$, then 
	$[(x_0,\tau_0)]\in F$ by~\eqref{closed_item1}. 
	We may thus assume that $[(x_0,\tau_0)]$ is an accumulation point of $F\cap \Gamma_2$,  
	and then $[(x_0,\tau_0)]\in F$ by~\eqref{closed_item3}--\eqref{closed_item4}. 
	This completes the proof. 
\end{proof}

\begin{remark}
	\normalfont 
	The method of proof of the fact that \eqref{closed_item3}--\eqref{closed_item4} in Proposition~\ref{closed} are sufficient for closedness of $F$, 
	that method carries over to a more general setting and leads to the following fact: 
	
	Let $Y$ be a 1st countable topological space with an increasing family of open subsets 
	$$\emptyset =D_0\subseteq D_1\subseteq\cdots \subseteq D_m=Y.$$
	For $j=1,\dots,m$ 
	denote $\Gamma_j:=D_j\setminus D_{j-1}$. 
	Then a subset $F\subseteq Y$ is closed if and only if the following conditions are satisfied: 
	\begin{enumerate}[{\rm(a)}]
		\item The set $F\cap \Gamma_j$ is closed in the relative topology of $\Gamma_j$ for $j=1,\dots,m$. 
		\item\label{b} If $1\le j_1<j_2\le m$, and $y\in\Gamma_{j_2}$ is an accumulation point of $F\cap \Gamma_{j_1}$, then $y\in F$. 
	\end{enumerate}
	Thus the main point of Proposition~\ref{closed} is that every closed subset $F\subseteq(X\times\widehat{A})/\sim$ satisfies the conditions \eqref{closed_item2}--\eqref{closed_item4}, which are much sharper than \eqref{b} above.  
\end{remark}

\begin{corollary}\label{closed_cor}
	The set $\Gamma_1$ is the largest open subset of $(X\times\widehat{A})/\sim$ whose relative topology is Hausdorff. 
	The set $\Gamma_2$ is the largest open subset of $((X\times\widehat{A})/\sim)\setminus\Gamma_1$ whose relative topology is Hausdorff, and $\Gamma_2'$ and $\Gamma_2''$ are the connected components of $\Gamma_2$. 
\end{corollary}

\begin{proof}
	These assertions follow by Proposition~\ref{closed} and Lemma~\ref{sep}. 
\end{proof}

\begin{corollary}\label{limit_sets}
We have the partition $\Gamma_1\sqcup(\Gamma_2'\sqcup\Gamma_2'')\sqcup\Gamma_3=(X\times\widehat{A})/\sim$ and the homeomorphisms 
\allowdisplaybreaks
\begin{align*}
\Psi_1 &\colon (0,\infty)^2\to \Gamma_1, 
\quad (r_1,r_2)\mapsto \bigl((r_1\TT)\times(r_2\TT)\bigr)\times\RR \\
\Psi_2' &\colon (0,\infty)\times(\RR/\ZZ)\to \Gamma_2', 
\quad (r_1,\tau+\ZZ)\mapsto \bigl((r_1\TT)\times\{0\}\bigr)\times(\tau+\ZZ) \\
\Psi_2'' &\colon (0,\infty)\times(\RR/\theta\ZZ)\to \Gamma_2'', 
\quad (r_2,\tau+\theta\ZZ)\mapsto \bigl(\{0\}\times(r_2\TT)\bigr)\times(\tau+\theta\ZZ), \\
\Psi_3&\colon \RR\to\Gamma_3,
\quad \tau\mapsto\{(0,0,\tau)\}.
\end{align*}
Moreover, 
\allowdisplaybreaks
\begin{align*}
\Lc'((X\times\widehat{A})/\sim)
=
& 
\{\Psi_1(r_1,r_2)\mid (r_1,r_2)\in(0,\infty)^2\} \\
&\sqcup\{\Psi_2'(\{r_1\}\times(\RR/\ZZ))\mid r_1\in(0,\infty)\} \\
&\sqcup\{\Psi_2''(\{r_2\}\times(\RR/\theta\ZZ))\mid r_2\in(0,\infty)\} \\
&\sqcup\{\Gamma_3\} \\
&\sqcup(\{\Psi_3(\tau+\ZZ)\mid \tau\in\RR\}
\cup\{\Psi_3(\tau+\theta\ZZ)\mid \tau\in\RR\}) \\
&\subseteq\Cl((X\times\widehat{A})/\sim) 
\end{align*}
and 
\begin{align*}
\Mc\Lc((X\times\widehat{A})/\sim)
=
& 
\{\Psi_1(r_1,r_2)\mid (r_1,r_2)\in(0,\infty)^2\} \\
&\sqcup\{\Psi_2'(\{r_1\}\times(\RR/\ZZ))\mid r_1\in(0,\infty)\} \\
&\sqcup\{\Psi_2''(\{r_2\}\times(\RR/\theta\ZZ))\mid r_2\in(0,\infty)\} \\
&\sqcup\{\Gamma_3\} \\
&\subseteq\Cl((X\times\widehat{A})/\sim). 
\end{align*}
Moreover the mapping
$$\Lambda\colon [0,\infty)^2\to\Mc\Lc((X\times\widehat{A})/\hskip-2pt\sim),\ 
\Lambda(r_1,r_2):=
\begin{cases}
\Psi_1(r_1,r_2) & \text{if }r_1r_2\ne 0,\\
\Psi_2'(\{r_1\}\times(\RR/\ZZ)) & \text{if }r_1\ne 0=r_2,\\
\Psi_2''(\{r_2\}\times(\RR/\theta\ZZ)) & \text{if }r_1=0\ne r_2,\\
\Gamma_3 & \text{if }r_1=r_2=0,
\end{cases}
$$
is a homeomorphism. 
\end{corollary}

\begin{proof}
The fact that the mappings $\Psi_1,\Psi_2',\Psi_2'',\Psi_3$ are homeomorphisms follows by Lemma~\ref{sep} and its preceding discussion, 
along with Lemma~\ref{param3}. 

The description of the set of nonempty closed limit sets $\Lc'((X\times\widehat{A})/\sim)$ follows by the description of the topology of $(X\times\widehat{A})/\sim$ given in Proposition~\ref{closed}. 
Then the set of the maximal limit sets $\Mc\Lc((X\times\widehat{A})/\sim)$ can easily be infered from the description of $\Lc'((X\times\widehat{A})/\sim)$. 
We also have 
The fact that the mapping $\Lambda$ is a homeomorphism follows 
by the fact that $A=\lim\limits_{n\to\infty}A_n$ in $\Cl((X\times\widehat{A})/\sim)$ with respect to the Fell topology $\tau_s$ 
if and only if $A=\liminf\limits_{n\to\infty}A_n=\limsup\limits_{n\to\infty}A_n$, 
cf. Definition~\ref{deflim} and \cite[2.1, 3.1, and 3.8]{Fl63}.  
(See also \cite{NS96}.)
\end{proof}

\begin{corollary}\label{Glimm}
If $x=\begin{pmatrix}
z_1 \\ z_2
\end{pmatrix},y=\begin{pmatrix}
w_1 \\ w_2
\end{pmatrix}\in \CC^2=X$ and $\tau,\sigma\in\RR\simeq\widehat{A}$,  
then the following assertions are equivalent: 
\begin{enumerate}[{\rm(i)}]
	\item\label{Glimm_item1} 
	For every $f\in\Ccb((X\times\widehat{A})/\sim)$ we have 
	$f([(x,\tau)])=f([(y,\sigma)])$. 
	\item\label{Glimm_item2} 
	We have $\vert z_j\vert=\vert w_j\vert$ for $j=1,2$. 
	\item\label{Glimm_item3} There exists $L\in \Mc\Lc((X\times\widehat{A})/\sim)$ with $[(x,\tau)],[(y,\sigma)]\in L$. 
	\item\label{Glimm_item4} The points $[(x,\tau)],[(y,\sigma)]\in (X\times\widehat{A})/\sim$ do not have disjoint neighbourhoods. 
\end{enumerate}
\end{corollary}

\begin{proof}
``\eqref{Glimm_item1}$\implies$\eqref{Glimm_item2}''
It follows by \eqref{tab} that the functions 
$$\varphi_1,\varphi_2\colon ((X\times\widehat{A})/\sim)\to[0,\infty),\quad 
\varphi_j([(x,\tau)]):=\vert z_j\vert \text{ for }j=1,2$$
are well defined, and they are also continuous since their lifts from the quotient space $(X\times\widehat{A})/\sim$ to $X\times\widehat{A}$ are clearly continuous. 
Moreover, $f_j:=\ee^{-\varphi_j}\in\Ccb((X\times\widehat{A})/\sim)$ for $j=1,2$ and, if $f_j([(x,\tau)])=f_j([(y,\sigma)])$ then clearly $\vert z_j\vert=\vert w_j\vert$ for $j=1,2$. 

``\eqref{Glimm_item2}$\implies$\eqref{Glimm_item3}'' 
This follows by the description of 
$\Mc\Lc((X\times\widehat{A})/\sim)$ in Corollary~\ref{limit_sets}. 

``\eqref{Glimm_item3}$\implies$\eqref{Glimm_item4}'' 
If 
there exists a limit set $L\in \Mc\Lc((X\times\widehat{A})/\sim)$  with the property that $[(x,\tau)],[(y,\sigma)]\in L$, 
then the points $[(x,\tau)]$ and $[(y,\sigma)]$ have no disjoint neighbourhoods by \cite[Lemme 9]{Di68}.

``\eqref{Glimm_item4}$\implies$\eqref{Glimm_item1}'' 
This is straightforward and well known. 
\end{proof}

\begin{corollary}
	\label{M} 
	For every $\theta\in\RR\setminus\QQ$ the following assertions hold: 
	\begin{enumerate}[{\rm(i)}] 
		\item\label{M_item1} We have $\Orc(C^*(G_\theta))=1$. 
		\item\label{M_item2} We have $\MinPrimal(C^*(G_\theta))=\Sub(C^*(G_\theta))=\Glimm(C^*(G_\theta))$ 
		as topological spaces, and these spaces are homeomorphic to $[0,\infty)^2$. 
		\item\label{M_item3} The $C^*$-algebra $C^*(G_\theta)$ is quasi-standard. 
	\end{enumerate}
\end{corollary}

\begin{proof}
	\eqref{M_item1} This follows by Corollary~\ref{Glimm}. 
	
	\eqref{M_item2}--\eqref{M_item3} 
	Using Lemma~\ref{primals}, the equality 
	$\MinPrimal(C^*(G_\theta))=\Glimm(C^*(G_\theta))$ (as sets) 
	follows by Corollary~\ref{Glimm}(\eqref{Glimm_item1},\eqref{Glimm_item3}). 
	Moreover, by the description of the set 
	$\Mc\Lc((X\times\widehat{A})/\sim)$ in Corollary~\ref{limit_sets},  
	it follows that $\Mc\Lc((X\times\widehat{A})/\sim)$ is a closed subset of $\Cl((X\times\widehat{A})/\sim)$, 
	that is, $\Mc\Lc((X\times\widehat{A})/\sim)=\Mc\Lc^s((X\times\widehat{A})/\sim)$, 
	and then 
	$\MinPrimal(C^*(G_\theta))=\Sub(C^*(G_\theta))$ by Lemma~\ref{primals} again. 
	
	Using Remark~\ref{quasi}\eqref{quasi_item3}, we now obtain that $C^*(G_\theta)$ is quasi-standard, 
	and then $\MinPrimal(C^*(G_\theta))=\Sub(C^*(G_\theta))=\Glimm(C^*(G_\theta))$ 
	as topological spaces by Remark~\ref{quasi}\eqref{quasi_item2}. 
	It also follows by Corollary~\ref{Glimm}\eqref{Glimm_item2} that these topological spaces are homeomorphic to $[0,\infty)^2$. 
\end{proof}

\section{The primitive quotients of $C^*(G_\theta)$}

Since $G_\theta$ is a connected and simply connected solvable Lie group whose roots are purely imaginary, the primitive quotients  $C^*(G_\theta)/\P$ for arbitrary $\P\in\Prim\,C^*(G_\theta)$ are simple $C^*$-algebras by \cite[Th. 2]{Pu73}. 
In this section we describe these simple $C^*$-algebras in terms of the Williams parameterization, and to this end we discuss separately the three parts of the partition 
\begin{equation}
\label{partition}
(X\times\widehat{A})/\sim\ =\Gamma_1\sqcup\Gamma_2\sqcup\Gamma_3. 
\end{equation}
(Compare \eqref{closed_proof_eq1}.) 

We recall from Lemma~\ref{param1}\eqref{param1_item3} that 
$$(\forall x\in X=\CC^2)(\forall \tau\in\widehat{A})\quad 
\Phi(x,\tau)=\Ker(\widetilde{M}^\alpha_x\rtimes\Ind_{A_x}^A(\tau\vert_{A_x})).$$

\begin{proposition}\label{primquot}
For every $[(x,\tau)]\in (X\times\widehat{A})/\sim$ the following assertions hold: 
\begin{enumerate}[{\rm(i)}]
	\item\label{primquot_item1}
	If $[(x,\tau)]\in\Gamma_1$ then $C^*(G_\theta)/\Phi(x,\tau)\simeq \A_\theta\otimes \Kc(L^2(\TT))$. 
	\item\label{primquot_item2}
	If $[(x,\tau)]\in\Gamma_3$ then $C^*(G_\theta)/\Phi(x,\tau)\simeq \CC$. 
	\item\label{primquot_item3}
	If $[(x,\tau)]\in\Gamma_3$ then $C^*(G_\theta)/\Phi(x,\tau)\simeq \CC$. 
\end{enumerate}
\end{proposition}

\begin{proof}
\eqref{primquot_item1} 
If $[(x,\tau)]\in\Gamma_1$ then one has $x\in(\CC^\times)^2$ and $A_x=\{0\}$, 
hence $\Ind_{A_x}^A(\tau\vert_{A_x})$ is the regular representation $\lambda\colon \RR\to\Bc(L^2(\RR))$ and  
$L^2(A,\tau\vert_{A_x})=L^2(\RR)$. 
Moreover, the $*$-representation 
$\widetilde{M}^\alpha_x\colon \Cc_0(X)\to\Bc(L^2(A,\tau\vert_{A_x}))$ 
takes on the form 
\begin{equation}\label{primit_case1_eq1}
\widetilde{M}^\alpha_x\colon \Cc_0(\CC^2)\to\Bc(L^2(\RR)), 
\quad (\widetilde{M}^\alpha_x(\xi)\varphi)(t)=\xi(\ee^{\ie tD_\theta}x)\varphi(t).
\end{equation}
We now write $x=\begin{pmatrix}z_1 \\ z_2 \end{pmatrix}$ 
and we denote $r_j:=\vert z_j\vert\in(0,\infty)$. 
Since $\theta\in\RR\setminus\QQ$, 
the set $\{\ee^{\ie tD_\theta}x\mid t\in\RR\}$ is dense in $r_1\TT\times r_2\TT$, and it then easily follows by \eqref{primit_case1_eq1} that 
$$\Jg_{r_1,r_2}:=\Ker \widetilde{M}^\alpha_x=\{\xi\in\Cc_0(\CC^2)\mid \xi\vert_{r_1\TT\times r_2\TT}=0\}.$$
It is clear that the ideal  $\Jg_{r_1,r_2}$ is invariant to the action of $A=(\RR,+)$ on $\Cc_0(\CC^2)$ via~\eqref{alphatheta}, and we will now show that 
\begin{equation}\label{primit_case1_eq2}
\Jg_{r_1,r_2}\rtimes_\alpha \RR=\Phi(x,\tau).
\end{equation}
In fact, one has a natural $A$-equivariant $*$-isomorphism 
$\Jg_{r_1,r_2}\simeq\CC_0(\CC^2\setminus(r_1\TT\times r_2\TT))$. 
As a special case of~\eqref{exact}, we then obtain the short exact sequence 
\begin{equation}\label{primit_case1_eq3}
0\to\Jg_{r_1,r_2}\rtimes \RR\hookrightarrow \Cc_0(\CC^2)\rtimes_\alpha \RR
\to\Cc(r_1\TT\times r_2\TT)\rtimes_\alpha \RR\to 0.
\end{equation}
Here $\Cc(r_1\TT\times r_2\TT)\rtimes_\alpha \RR\simeq \A_\theta\otimes\Kc(L^2(\TT))$ by  \eqref{Gr80Th.4.1_ex1_eq1}, 
and on the other hand $\A_\theta$ is simple if $\theta\in\RR\setminus\QQ$. 
(See for instance \cite[Th. VI.1.4]{Da96}.) 
Therefore  $\Cc(r_1\TT\times r_2\TT)\rtimes_\alpha \RR$ is a simple $C^*$-algebra and then, by \eqref{primit_case1_eq3} along with Lemma~\ref{param0}\eqref{param0_item2} we obtain 
$\Jg_{r_1,r_2}\rtimes \RR=\Ker(\widetilde{M}^\alpha_x\rtimes\Ind_{A_x}^A(\tau\vert_{A_x}))$, 
that is, \eqref{primit_case1_eq2} holds true. 

Moreover, 
using \eqref{Gr80Th.4.1_ex1_eq1} and \eqref{primit_case1_eq3} again, 
we obtain the short exact sequence 
\begin{equation}\label{primit_case1_eq4}
0\to\Phi(x,\tau)\hookrightarrow C^*(G_\theta)
\to \A_\theta\otimes\Kc(L^2(\TT))\to 0
\end{equation}
which shows that $C^*(G_\theta)/\Phi(x,\tau)\simeq \A_\theta\otimes \Kc(L^2(\TT))$, as claimed. 

\eqref{primquot_item2} 
Write $x=\begin{pmatrix}z_1 \\ z_2 \end{pmatrix}$. 
Since $[(x,\tau)]\in\Gamma_2$, we have either $z_1=0\ne z_2$ or $z_1\ne0=z_2$. 
We then obtain by Remark~\ref{param2.6} $*$-isomorphisms 
$C^*(G_\theta)/\Phi(x,\tau)\simeq \Kc(L^2(\RR/\ZZ))$ 
or $C^*(G_\theta)/\Phi(x,\tau)\simeq \Kc(L^2(\RR/\theta\ZZ))$, respectively. 

\eqref{primquot_item3} 
In fact, if $[(x,\tau)]\in\Gamma_3$, then $x=0\in\CC^2$, 
and then Remark~\ref{param2.7} gives a $*$-isomorphism
$C^*(G_\theta)/\Phi(0,\tau)\simeq \CC$
for every $\tau\in\RR$, as claimed.  
\end{proof}

\begin{remark}
	\label{quotients}
\normalfont
It follows by Proposition~\ref{primquot} that 
	if $\theta\in\RR\setminus\QQ$, then every primitive quotient of $C^*(G_\theta)$ is $*$-isomorphic either to $\A_\theta\otimes\Kc(L^2(\TT))$, or to $\Kc(L^2(\TT))$, or to~$\CC$. 
\end{remark}

\subsection*{Acknowledgements} 
We wish to thank Professor \c Serban Str\u atil\u a for encouragement and to Professor Douglas Somerset for several interesting remarks on an earlier version of this paper. 

The research of the second-named author was supported by
a grant of the Ministry of Research, Innovation and Digitization, CNCS/CCCDI –
UEFISCDI, project number PN-III-P4-ID-PCE-2020-0878, within PNCDI III.

\end{document}